\documentclass[11pt]{amsart}
\usepackage[centering]{geometry} 
\usepackage{hyperref}
\usepackage{graphicx}
\usepackage{mathrsfs}
\usepackage{amssymb}
\usepackage{epstopdf}
\usepackage{amscd}
\usepackage{comment}
\usepackage[english]{babel}
\usepackage{tikz}
\usepackage{latexsym}
\usetikzlibrary{calc}
\usetikzlibrary{shapes.arrows}
\usetikzlibrary{shapes.geometric}

\newtheorem{theore}{Theorem}
\newtheorem{propositio}{Proposition}
\newtheorem{theorem}{Theorem}[section]

\newtheorem{proposition}[theorem]{Proposition}
\newtheorem{corollary}[theorem]{Corollary}
\newtheorem{claim}[theorem]{Claim}
\newtheorem{lemma}[theorem]{Lemma}
\newtheorem{remark}[theorem]{Remark}

\newtheorem{definition}[theorem]{Definition}

\newtheorem*{theorem-non}{Theorem}

\DeclareMathOperator{\rk}{rank}

\DeclareMathOperator{\gal}{Gal}
\DeclareMathOperator{\GL}{GL}

\DeclareMathOperator{\SL}{SL}
\DeclareMathOperator{\End}{End}
\DeclareMathOperator{\tr}{tr}
\DeclareMathOperator{\Aff}{Aff}
\DeclareMathOperator{\Stab}{Stab}

\author{Federico Buonerba}
\address{\tiny{Federico Buonerba\newline Courant Institute of Mathematical Sciences,
 New York University, 
 251 Mercer Street, 
 New York, NY 10012, USA}}
 \email{buonerba@cims.nyu.edu}

\author{Fedor Bogomolov} 
\address {\tiny{Fedor Bogomolov \newline Courant Institute of Mathematical Sciences,
 New York University, 
 251 Mercer Street, 
 New York, NY 10012, USA\newline
 National Research University Higher School of Economics, Russian Federation,
AG Laboratory, HSE, 6 Usacheva str., Moscow, Russia, 119048}}
\email{bogomolo@cims.nyu.edu}

\author{Nikon Kurnosov}
\address {\tiny{Nikon Kurnosov\newline
 National Research University Higher School of Economics, Russian Federation,
AG Laboratory, HSE, 6 Usacheva str., Moscow, Russia, 119048 \newline
Department of Mathematics,
University of Georgia,
1023 D. W. Brooks Drive,
Athens, GA 30602, USA}}
\email{nikon.kurnosov@gmail.com}

\title{Classifying $VII_0$ surfaces with $b_2=0$ via group theory}
\begin{document}
\begin{abstract}
 We give a new proof of Bogomolov's theorem, that the only $VII_0$ surfaces with $b_2=0$
 are, up to an \'etale cover, those constructed by Hopf and Inoue.
 The proof follows the strategy of the original one, but it is of purely
 group-theoretic nature.

\end{abstract}

\maketitle

Our aim is to provide a new proof of the well known:
\begin{theore}\cite[Bogomolov]{1}\label{bogomolov}
 Every $VII_0$ surface with $b_2=0$ has an \'etale cover that carries a holomorphic foliation.
 Therefore an \'etale cover of $X$ is either a Hopf or a Inoue surface.
\end{theore}
The first part of the proof in \textit{op.cit.}, which is the content of our Section \ref{three},
aims at constructing an affine structure on such surfaces, as well as
establishing its uniqueness.
Both the existence and uniqueness statements follow from the vanishing
of certain cohomology groups, which could be naturally thought of as obstruction spaces.
This structure provides two objects: a representation of the fundamental group of the surface into the affine group 
$\rho: \Gamma:=\pi_1(X)\to \Aff(2,\mathbf C)$,
with corresponding linearization $l\rho$ into $\GL_2(\mathbf C)$;
and a $\rho$-equivariant local biholomorphism $p_{\text{aff}}:V:=\tilde X\to \mathbf C^2$.\\
The automorphism group of the field $\mathbf C$ which we denote as $\gal (\mathbf {C/Q})$ acts naturally on the set of
flat vector bundles over $X$. 
It can be shown that indeed the tangent bundle is stable under this action.
Such stability follows from the fact that these surfaces have very few vector bundles with at least one non-vanishing
cohomology group, and that cohomology with locally constant coefficients is not altered by the Galois action.
It follows that the linear representations, obtained by twisting $l\rho$ with field automorphisms, 
are conjugated to each other.\\
This observation imposes very strong arithmetic restrictions on such representation, namely that
the image $l\rho(\Gamma)$ can be conjugated into $\GL_2(\mathbf Q)$ or into $H_2(F)$, where the
latter is the unit group in a quaternion algebra over a quadratic field $F$.
Crucially, however, the quadratic extension $F/\mathbf Q$ must be real. This follows from a very particular fact
 about $VII_0$ surfaces, namely that dim$(H^1(X,\mathbf C(l\rho)))=1$ - where $\mathbf C(l\rho)$ is the flat bundle 
 induced by $l\rho$.
This fact is also crucial in deriving that the image $\rho(\Gamma)$ is contained in 
the real affine subgroup $\Aff(2,\mathbf R)\subset \Aff(2,\mathbf C)$. 
The group $\Aff(2,\mathbf R)$ above leaves invariant a linear subspace $R \xrightarrow{\sim} \mathbf R^2\subset \mathbf C^2$,
and there is a two-dimensional contractible Lie subgroup $G_K\subset \Aff(4,\mathbf R)$
of real affine transformations acting on $\mathbf C^2$,
which is free on $\mathbf C^2\setminus R$ and fixes $R$ pointwise.
The quotient of $\mathbf C^2\setminus R$ by its
action is a Moebius band isomorphic to the space $Gr(1,2)$ of lines
in $\mathbf R^2$, which carries a natural angle map $s:Gr(1,2)\to S^1$.
All those structures are inherited by $V$ and since they are $\Gamma$-invariant,
also by $X$.\\
The pre-image of the real subspace $R$ under the projection $p_{\text{aff}}$
is a, possibly empty, closed 2-dimensional real submanifold $V_R\subset V$,.
The action of $G_K$ on $V\setminus V_R$ is free, while trivial on $V_R$.
The image of $V_R$ in $X$ is a compact two-dimensional manifold
with affine structure, {\it i.e} 
a possibly empty finite union of compact
tori or Klein bottles.
Thus the quotient $S:=(V\setminus V_R)/G_K$ is an open Riemann surface 
with a $\rho$-equivariant local biholomorphism $S\to Gr(1,2)$.
Diagram \ref{diagrammmm} consists of maps commuting
with the action of $\Gamma$, and summarizes the situation.\\
We thus have two possible radically different cases:
\begin{enumerate}
 \item $ S\to Gr(1,2)$ is a cyclic covering, which we call classical case.
 \item $ S\to Gr(1,2)$ is just a local isomorphism, hence rather wild, which we
 call pathological case.
\end{enumerate}
After this was achieved the original strategy, first sketched in \cite{1} and then expanded on in \cite{9},
aimed at deriving a contradiction profitting from the interplay between: the strong arithmetic
constraints we have on the affine representation $\rho$, and; the analysis
of geometric information
arising from the existence of a complex affine structure $p_{\text{aff}}$ with the above properties.
Unfortunately this derivation in \cite{9} is rather long, 
and contains complicated topological arguments.
Our contribution is to replace such a proof by a much simpler argument.\\
The case $V_R=\emptyset$ contains the main innovation: a combination of a group-theoretic argument, 
together with an analysis of the
tree structure of the topological Stein quotient $S_s$ of $\tilde S$, corresponding
to the induced angle projection $\tilde S\to \tilde Gr(1,2)\to \mathbf R$.
Though the Stein quotient $S_s$ is potentially non-Hausdorff,
we can use the fact that $\Gamma $ acts densely on $S$,
to conclude that if $V_R$ is empty, then the map 
$\tilde S\to \tilde Gr(1,2)$ must be an embedding.
This implies that $\Gamma $ is realized as a discrete 
subgroup of $\tilde A(2,\mathbf R)$. A contradiction is derived by purely group theoretic means.\\
In case $V_R\neq \emptyset$, we venture into a detailed study of the induced affine structure on
$V_R$, finally leading to a contradiction. 

It is interesting to remark that the key player, in the story, is our two dimensional
Lie group $G_K$ acting on $V$. In fact, its orbits are holomorphic curves, but the
resulting foliation is not holomorphic.\\
We should also remark that we can freely pass to finite-index subgroups of $\Gamma$, without
loss of generality. This is particularly important once we know that $l\rho$ has arithmetic values,
so then by Selberg's lemma, \cite{alp}, we can assume $l\rho(\Gamma)$ is torsion-free.\\
An analytic approach to this theorem appeared in \cite{8}, 
and was further completed in \cite{4} and \cite{10}.\\
Our paper is structured as follows: in the first section we construct the affine structure, and prove its
arithmetic properties. In the second section we study in details its geometry. In the third section we
prove that the case $V_R=\emptyset$ is impossible, and in the fourth we prove that $V_R\neq \emptyset$
is impossible.

\noindent
\textbf{Acknowledgements} 
Fedor Bogomolov and Nikon Kurnosov acknowledge that the article
was prepared within the framework of a subsidy granted to the HSE by the Government
of the Russian Federation for the implementation of the Global Competitiveness Program.
The second author was partially
supported by EPSRC programme grant EP/M024830, Simons Fellowship
and Simons travel grant.\\

\section{Affine structures on $VII_0$ surfaces with $b_2=0$ and no holomorphic foliations}
\label{three}
The content of this section is the proof of:
\begin{proposition}\label{bogomolovvv}
 Let $X$ be a $VII_0$ surface with $b_2=0$ and no holomophic foliation. Then $X$ carries
 an affine structure, which defines a representation $\rho:\pi_1(X)\to \Aff(2,\mathbf C)$.
 Furthermore, the image of its linearization $l\rho$ is conjugated to a subgroup of
 $\GL_2(\mathbf Q)$,
 or of the units $H_2(F)$ of a quaternion algebra over $F$, a real quadratic extension of $\mathbf Q$.
 \end{proposition}
 The structure of the section is as follows:
the proof is split into four main technical statements, namely Propositions \ref{affine}, \ref{Bombieri},
\ref{wonderful} and \ref{conjugation}; each such Proposition is proved in a sequence of simple Lemmas.\\
Examples of $VII_0$ surfaces, {\it i.e.} Hopf and Inoue surfaces,
are described in \cite{1}, \cite{55}, and \cite{7}. We want to show that there are no other examples,
which amounts to proving that surfaces
with the following properties do not exist:
\begin{itemize}\label{surfaceVII}
 \item  $c_1(X)$ is torsion - in particular $c_1^2=0$ - and $c_2(X)=0$.
 \item $b_2=0$, $H^0(X,K_X^n)=0$ for every $n\in \mathbf Z\setminus \{0\}$.
 \item It has infinite fundamental group, moreover $\dim H^1(X,\mathbf C)=1$.
 \item It carries no holomorphic
foliations, i.e. $H^0(X,\mathscr L\otimes \Omega^1)=0$ for every line bundle $\mathscr L$.
\item The kernel of $\pi_1(X)\to \pi_1(X)/[\pi_1(X),\pi_1(X)]$ is infinitely generated (Bombieri \cite{55}).
\item Every finite \'etale cover of $X$ has all the previous properties.
\end{itemize}
We will denote $\Gamma:=\pi_1(X)$.\\
In this section we follow the strategy of \cite{1} rather closely.
\begin{proposition}\label{affine}
The tangent bundle $TX$ admits a unique flat, torsion-free connection.
\end{proposition}
\begin{proof}
 The following remark is used systematically in the proof:
 the triviality of Chern numbers implies that the Euler characteristic of any tensor bundle $\mathscr F$ 
 is trivial; hence
 the existence of a non-trivial element in $H^1(\mathscr F)$, implies the existence of
 a non-trivial holomorphic section of $\mathscr F$ or $K_X\otimes \mathscr F^{\vee}$.\\
 Let us consider the obstructions to the existence of an affine structure on $X$.
 \begin{itemize}
  \item The obstruction to the existence of a holomorphic connection on $TX$ is in\\ 
  $H^1(\End(TX)\otimes \Omega^1)$
  \item The obstruction to the flatness of the connection is in $H^0(K_X\otimes \End(TX))$
  \item The obstruction to the uniqueness of the connection is in $H^0(\End(TX)\otimes \Omega^1)$
 \end{itemize}
This is a standard fact which is directly translated from differtial geometry.\\ 
We shall prove that all these cohomology groups vanish.
\begin{lemma}\label{vanishing1}
 For every $n\in \mathbf Z$, $n\neq 0$, we have $H^0(K_X^n\otimes \End(TX))=0$.
\end{lemma}
\begin{proof}
If $s$ is a section, it defines a morphism $s:TX\to TX\otimes K_X^n$.
Its determinant is a section of $K_X^{4n}$, which must be zero.
Hence, if $s$ is non-trivial, its kernel is generically of rank $1$ and
defines a holomorphic foliation, which is absurd.
\end{proof}
\begin{lemma}
 $H^0(\Omega^1\otimes \End(TX))=0$
\end{lemma}
\begin{proof}
Let us assume it is non-zero, and hence there is a non-zero section. This tautologically
defines a non-zero morphism $s:TX\to \End(TX)$. 
It is injective, otherwise the kernel defines a foliation. 
Composing $s$ with the natural trace
map $\tr:\End(TX)\to \mathscr O_X$, we obtain a morphism $\text{tr}\circ s:TX\to \mathscr O_X$,
which must be trivial since otherwise its kernel would again define a foliation.
Next, consider the composition of $s\wedge s:K_X^{-1}=TX\wedge TX\to \End(TX)\wedge \End(TX)$ with 
the commutator $c:\End(TX)\wedge \End(TX)\to \End(TX)$. By Lemma \ref{vanishing1}, it must
be zero. It follows that the image of $s$ is a rank $2$, traceless subbundle, which generates a commutative sub-algebra of
$\End(TX)$. Such subalgebra is necessarily one-dimensional, and corresponding representation is not irreducible,
contradicting our assumption that no \'etale cover of $X$ carries non-trivial holomorphic foliations.
\end{proof}
At this point, we proved the vanishing of the last two obstructions mentioned above.
On the other hand, since $\chi(X)=0$, if $h^1(\End(TX)\otimes \Omega^1)\neq 0$
then by Riemann-Roch and Serre duality, $h^0(\End(TX)\otimes \Omega^1)$ is non-zero 
(we are using $K_X\otimes TX\xrightarrow{\sim} \Omega^1$ to identify $h^0$ and $h^2$)
which is impossible. 
\end{proof}
The connection just constructed defines an affine structure on $X$, meaning a local biholomorphism 
from the universal cover
\begin{equation}\label{developingmap}
 p_{\text{aff}}:\tilde{X}\to \mathbf C^2
\end{equation}
which is equivariant for a representation $\rho:\Gamma:=\pi_1(X)\to \Aff(2,\mathbf C)$,
with corresponding linearization $l\rho:\Gamma\to \GL_2(\mathbf C)$. 
We are going to use systematically the fact that $\rho(\Gamma)$-invariant objects on $\mathbf C^2$
define naturally similar objects on the compact surface $X$:
first we lift them under $p_{\text{aff}}$ to $\Gamma$-invariant objects on the universal covering
$V$, and then descend them to $X$.
The first example of implementation of this general principle is:
\begin{lemma}\label{Gnotcyclic}
 The action of $Im(l\rho)$ is irreducible. In particular, such group is not cyclic.
\end{lemma}
\begin{proof}
 If it were then there would be a direction, say $(x=0)\subset \mathbf C^2$, which is invariant under
 $Im(l\rho)$. Consider the 1-form $dx$. 
 It is clearly invariant under translations, and by definition the action of $Im(l\rho)$ on $dx$
 is defined by a character $\chi$.
 Then the pull back $p_{\text{aff}}^*(dx)$ defines a $\Gamma$-invariant holomorphic foliation on $V$,
 which descends to $X$ - contradiction.
\end{proof}
We will extract informations by studying the orbit of $l\rho$ under the natural action of $\gal(\mathbf {C/Q})$,
by which we mean the full automorphism group of $\mathbf C$ as a field.
Namely, if a bundle $E$ is defined by a representation $\eta:\Gamma\to \GL_n(\mathbf C)$,
we will denote by $E^{\sigma}$ the one defined by $\sigma \circ \eta$, for $\sigma \in \gal(\mathbf {C/Q})$.
For $\mathscr F$ a local system,
we have exact sequences in the analytic topology:
 \begin{equation}\label{eq1}
  0\to \mathbf C(\mathscr F)\to \mathscr {O(F)}\to d\mathscr {O(F)}\to 0
 \end{equation}
 
 \begin{equation}\label{eq2}
  0\to d\mathscr {O(F)}\to \Omega^1(\mathscr F)\to K_X\otimes \mathscr F\to 0
 \end{equation}
Where $\mathbf C(\mathscr F)$ is the sheaf of locally constant sections, $\mathscr{O(F)}$ the sheaf of
holomorphic sections of the coherent sheaf constructed from $\mathscr F$, and $d$ denotes the usual holomorphic differential.
\begin{proposition}\cite[Bombieri]{55}\label{Bombieri}
$ K_X^{\sigma}\xrightarrow{\sim} K_X $ for every $\sigma$.
\end{proposition}
\begin{proof}
Let $\mathscr F$ be a local system.
\begin{lemma}\label{van1}
 If $\mathscr F\neq \mathscr {O}_X, K_X$ is of rank $1$, then $h^i(X,\mathscr {O(F)})=0$ for every $i$.
\end{lemma}
\begin{proof}
This is the absence of compact curves, plus the triviality of Chern classes and Riemann-Roch.
\end{proof}
\begin{lemma}\label{van2}
 If $\mathscr F$ has rank $1$, then $h^0(d\mathscr{O(F)})=0$.
\end{lemma}
\begin{proof}
This follows by taking global sections in the sequence \ref{eq2}.
\end{proof}
\begin{lemma}\label{van3}
If $\mathscr F\neq \mathscr O_X, K_X^{-1}$ has rank $1$ then $h^i(d\mathscr {O(F)})=0$ for every $i$. 
\end{lemma}
\begin{proof}
We can apply sequence \ref{eq2}, together with Lemma \ref{van1} for $K_X\otimes \mathscr F$, to
deduce $h^i(d\mathscr {O(F)})=h^i(\Omega^1(\mathscr F))$. If, for some $i$ we have $h^i(\Omega^1(\mathscr F))\neq 0$,
then by Riemann-Roch it holds for $i=0$ or $i=2$. But then we get a foliation, which is absurd.
\end{proof}
\begin{lemma}
 $h^1(\mathbf C(K_X))=1$
\end{lemma}
\begin{proof}
By Lemma \ref{van3}, we know $h^i(d\mathscr {O(F)})=0$ for all $i$. By the sequence \ref{eq1}
$h^1(\mathbf C(K_X))=h^1(\mathscr O(K_X))=1$.
\end{proof}
The above lemmas now merge to conclude the proof: if $K_X^{\sigma}\neq K_X$, by Lemma \ref{van1}
$h^i(\mathscr O(K_X^{\sigma}))=0$. Then by the sequence \ref{eq1} we get 
$$h^0(d\mathscr O(K_X^{\sigma}))=h^1(\mathbf C(K_X^{\sigma}))=h^1(\mathbf C(K_X))=1$$
which is absurd by lemma \ref{van2}.

\end{proof}
Now we turn our attention to the Galois action on $TX$. As one imagines,
\begin{proposition}\label{wonderful}
 $TX^{\sigma}\xrightarrow{\sim} TX$. In particular there exists a representation 
 $\xi:\gal(\mathbf{C/Q})\to \GL_2(\mathbf C)$ such that
 $l\rho^{\sigma}=\xi_{\sigma}\cdot l\rho \cdot \xi_{\sigma}^{-1}$.
\end{proposition}
\begin{proof}
 Let us start by observing that, for any local system $\mathscr F$,
 the existence of homotheties gives $h^0(\mathscr F\otimes \mathscr F^{\vee})=h^0(\End(\mathscr F))\geq 1$. 
 We have, again,
 a series of lemmas.
 \begin{lemma}\label{onedimens}
  $h^1(\mathbf C(TX))=h^0(d\mathscr O(TX))= 1$
 \end{lemma}
\begin{proof}
The first equality follows from sequence \ref{eq1} and the vanishing of $h^i(\mathscr O(TX))$ due
to the absence of foliations and Riemann-Roch. Observe that the absence of foliations implies that $TX$ is simple -
in particular any endomorphism of $TX$ must be a multiple of the identity,
$h^0(\End(TX))=1$ - and since $h^0(K_X\otimes TX)=0$,
we conclude via sequence \ref{eq2}.
\end{proof}
Therefore: $h^1(\mathbf C(TX^{\sigma}))= 1$ for every $\sigma\in \gal(\mathbf{C/Q})$.
By the sequence \ref{eq1}, at least one of the following happens:
\begin{itemize}
 \item $h^0(d\mathscr {O}(TX^{\sigma}))\geq 1$
 \item $h^1(\mathscr {O}(TX^{\sigma}))\geq 1$
\end{itemize}
The first option implies that, by the sequence \ref{eq2}, $h^0(\Omega^1\otimes TX^{\sigma})\geq 1$
so then there exists a morphism $TX\to TX^{\sigma}$ which must be an isomorphism.\\
The second option, along with Riemann-Roch, leads to a further alternative:
\begin{itemize}
 \item $h^0(\mathscr O(TX^{\sigma}))\geq 1$
 \item $h^2(\mathscr O(TX^{\sigma}))=h^0(K_X\otimes \mathscr O(TX^{\sigma})^{\vee})\geq 1$
\end{itemize}
In the first case, since $h^0(\mathbf C(TX^{\sigma}))=0$, the sequence \ref{eq1} implies
$h^0(d\mathscr O(TX^{\sigma}))\geq 1$, while the sequence \ref{eq2} gives an isomorphism $TX\to TX^{\sigma}$.\\
In the second case, by the sequence \ref{eq1} and the identity $(TX^{\sigma})^{\vee}=(\Omega^1)^{\sigma}$,
we have one more alternative:
\begin{itemize}
 \item $h^0(\mathbf C(K_X\otimes (\Omega^1)^{\sigma}))\geq 1$
 \item $h^0(d\mathscr O(K_X\otimes (\Omega^1)^{\sigma}))\geq 1$
\end{itemize}
In the first case Proposition \ref{Bombieri} gives $h^0(\mathbf C(K_X\otimes (\Omega^1)^{\sigma}))=
h^0(\mathbf C(K_X\otimes \Omega^1)^{\sigma})=h^0(\mathbf C(K_X\otimes \Omega^1))\geq 1$
and hence a foliation.\\
In the second case, the sequence \ref{eq2} gives $h^0(K_X\otimes \Omega^1\otimes (\Omega^1)^{\sigma})\geq 1$,
which defines an isomorphism $TX\xrightarrow{\sim} K_X\otimes (\Omega^1)^{\sigma}$. Taking determinants, however,
gives $K_X^{\vee}\xrightarrow{\sim} K_X^2\otimes K_X^{\sigma}\xrightarrow{\sim} K_X^3$ - the last
isomorphism being Bombieri's \ref{Bombieri} - which is
again absurd since $K_X$ is not torsion.\\
This proves that $TX\xrightarrow{\sim} TX^{\sigma}$ for every $\sigma$.
Since $TX$ is built from $\rho$, the representations $l\rho^{\sigma}$ are isomorphic to each other,
and the existence of $\xi$ follows.
\end{proof}
We wish to employ Proposition \ref{wonderful} to derive concrete arithmetic restrictions on
$l\rho$. Quite generally we have:
\begin{proposition}\label{conjugation}
 Let $\eta: \Gamma\to \GL_2(\mathbf C)$ be a representation, such that $\eta^{\sigma}$ is conjugated
 to $\eta$ for every $\sigma\in \gal(\mathbf {C/Q})$. Then the image of $\eta$ is either conjugated
 to a subgroup of $\GL_2(\mathbf Q)$, or to a subgroup of the invertible elements 
 in a quaternion algebra $H_2\subset M_2(\mathbf Q(\sqrt d))$ over $\mathbf Q$.
\end{proposition}
\begin{proof}
Since the representation is irreducible and 
$tr(g), det(g)\in \mathbf Q$ we obtain by standard argument that
the $\mathbf Q$-linear span of the image of
$\eta$ is a simple algebra $A$ over $\mathbf Q$.
We also have $A^{\sigma} = A$ for every $\sigma$. It implies that
$\sigma$ acts as either identity, or as a nontrivial involution
on $A$. In the first case $A= M(2,\mathbf Q)$.
In the second case it is a quaternion subalgebra
of $M(2,\mathbf Q (\sqrt d)$ with $d\in \mathbf Z$.
\end{proof}
We are going to distinguish two main cases:
\begin{itemize}
 \item $A\otimes \mathbf R= M(2,\mathbf R)$ when $d >0$.
 \item $A\otimes \mathbf R= H$
where $H$ is a standard algebra of quaternions over $\mathbf R$
when $d < 0$.
\end{itemize}
We are going to show now that the second case can not occur.
First we need a remark:

\begin{remark}\label{nonlinear}
The affine structure is not linear, since otherwise $X$ carries a holomorphic
foliation coming from the Euler vector field commuting with linear action of $\Gamma$.
In particular $\rho(\Gamma)$ is not contained in any subgroup conjugated to $\GL_2(\mathbf C)$
inside the affine group $A(2,\mathbf C)$.
\end{remark}
\begin{lemma}
$A\otimes \mathbf R\neq H$.
\end{lemma}
\begin{proof}
The action of $\rho(\Gamma)$ in $\mathbf C^2= \mathbf R^4$
is not linearzable by the previous remark, and hence the cohomology
element $\omega\in H^1(\pi_1, \mathbf C^2)=\mathbf C=\mathbf R^2= H^1(\pi_1, \mathbf R^4)$
defining $\rho$
is nonzero.
The centralizer of $\rho(\Gamma)$ in $M(4,\mathbf R)$
acts naturally and nontrivially on $H^1(\pi_1, \mathbf C^2)=\mathbf C=\mathbf R^2= H^1(\pi_1,\mathbf R^4)$,
viewed as real vector space.
Note that if $A\otimes \mathbf R = H$ then the representation
$l\rho$ remains irreduicble on $\mathbf R^4$ and hence is
quaternionic, by general theory of real representations.
The latter means that
the centralizer of $l\rho(\pi_1)$ in $M(4,\mathbf R)$ is also
isomoprhic to $H\subset M(4,\mathbf R)$, but with a different embedding of $H$, which denote by $H_r\subset
M(4,\mathbf R)$.
The algebra $H_r$
has to act linearly over $\mathbf R$ and nontrivilly on 
$H^1(\pi_1, \mathbf C^2)=\mathbf C=\mathbf R^2= H^1(\pi_1,\mathbf R^4)$
but $H_r$ has no nontrivial representations into $M(2,\mathbf R)$.
It provides a desired contradiction.
\end{proof}
Remark that the same argument works for any group $\Gamma$
with a representation $\rho$ in $\mathbf C^n$ and the property
that $\rho$ remains irreducible in $\mathbf R^{2n}= \mathbf C^n$
over real numbers. Then $H^1(\Gamma, \mathbf C^n)= H^1(\Gamma, \mathbf R^{2n})$ is $H$-module
and hence $H^1(\Gamma, \mathbf C^n)$ has to have even rank over $\mathbf C$.
Similar situation occurs for other pairs of fields $F\subset L$ instead of $\mathbf R,\mathbf C$.
This conludes the proof of Proposition
\ref{bogomolovvv}.

\section{Geometry of affine structures}\label{five}
In order to reach a contradiction using our affine structure, it will be important to study in detail the
geometry that governs it. This is the aim of the present section.
A substantial part of the presentation is inspired by
 \cite[chapter 3]{9}.\\
Recall that the affine structure on $X$ is defined
by a morphism $V:=\tilde{X}\to \mathbf C^2$, equivariant with respect
to a representation $\rho: \Gamma\to \Aff(2,\mathbf C)$.
This representation is uniquely defined by its linearization $l\rho$, whose
image we denote by $G:=Im(l\rho)$,
along with a cocycle $\omega\in H^1(\Gamma,\mathbf C^2_{\l\rho})$.
We already know the structure of $l\rho$,
namely its image is contained in $\GL_2(\mathbf R)<\GL_2(\mathbf C)<\GL_4(\mathbf R)$,
but something more can be said about $\rho$:
\begin{lemma}\label{realss}
 There exists a real $2$-dimensional subspace $R\subset \mathbf C^2$, invariant
 under $\rho(\Gamma)$.
\end{lemma}
\begin{proof}
First we set up some notation.
There is a natural embedding $e:\Aff(2,\mathbf C)\to \Aff(4,\mathbf R)$ obtained by
considering $\mathbf C$ as a real vector space, along with a linearized analogue
$le: \GL_2(\mathbf C)\to \GL_4(\mathbf R)$. 
A trivial inspection reveals that the centralizer of
$e(\GL_2( \mathbf R))$ is naturally isomorphic to $\GL_2(\mathbf R)$,
and will be referred to as $\GL_2(\mathbf R)_K$ in the sequel. 
Indeed, we can represent the elements of $e(\GL_2( \mathbf R))$, in some basis,
as matrices
\[
\begin{pmatrix}
aI & bI \\
cI & dI
\end{pmatrix}
\]
where $I$ is the $2\times 2$ identity, and $a,b,c,d\in \mathbf R$,
while those elements of $\GL_2(\mathbf R)_K$ can be written, in the same basis, as
\[
\begin{pmatrix}
A & 0 \\
0 & A
\end{pmatrix}
\]
with $A\in \GL_2(\mathbf R)$.
The intersection
$\GL_2(\mathbf R)_K\cap e(\GL_2( \mathbf C))$ is the center of $\GL_2( \mathbf C)$.\\
Now we can start the proof.
By Proposition \ref{bogomolovvv}
the image of $l\rho$ is inside $\GL_2(\mathbf R)$, and therefore the composite
$le\circ l\rho:\Gamma \to \GL_4(\mathbf R)$ 
is a reducible real representation, which splits as a direct sum $\rho_1\oplus \rho_2$
of isomorphic representations $\rho_i:\Gamma\to \GL_2(\mathbf R)$.
Observe that $\GL_2(\mathbf R)_K$ acts naturally on the set of such splittings of $le\circ l\rho$.
We want to improve this, to show that the whole
$e\circ \rho$ can be split into a sum of: a real affine representation valued in $\Aff(2,\mathbf R)$, and
a linear one valued in $\GL_2(\mathbf R)$ .\\
If we fix an $e(\GL_2( \mathbf R))$-invariant splitting 
$\mathbf R^4=\mathbf R^2\oplus \mathbf R^2$ then we have a
decomposition 
\begin{equation}
 H^1(\Gamma,\mathbf C^2_{\l\rho})\xrightarrow{\sim} H^1(\Gamma, \mathbf R^2_{\rho_1})\oplus H^1(\Gamma, \mathbf R^2_{\rho_2})
\end{equation}
where, by Lemma \ref{onedimens}, each space on the right is real $1$-dimensional;
correspondingly we obtain that the cocycle $\omega$, defining $\rho$, splits in a
pair $\omega_1\oplus \omega_2$ of real-valued cocycles.
The center of $\GL_2(\mathbf C)$ acts by scalar multiplication on
$H^1(\Gamma,\mathbf C^2_{\l\rho})$, and hence the whole $\GL_2(\mathbf R)_K$
acts transitively on $H^1(\Gamma, \mathbf R^2_{\rho_1})\oplus H^1(\Gamma, \mathbf R^2_{\rho_2})\setminus \{(0,0)\}$.
In particular, since $\omega\neq 0$, for some $e(\GL_2( \mathbf R))$-invariant
splitting of $\mathbf R^4$ we have $\omega_1=1$ and $\omega_2=0$, so then
the action of $\Gamma$ on $\mathbf R^4/\mathbf R^2_{\rho_1}$ is linear. The existence of $R$ follows.
\end{proof}
We continue by inspecting those elements commuting with the image of $\rho$, in the whole affine group.
From the existence of $R$, indeed by the previous proof, it follows that the image of $\rho$ lies in
the extension of $e(\GL_2(\mathbf R))$ by the subgroup of translations preserving $R$. This extension
is clearly isomorphic to
$\Aff(2,\mathbf R)$. We denote its image in
$\Aff(2,\mathbf C)$ by $\Aff(2,\mathbf R)_{\Delta}$, and by $G_K$ its centralizer 
in $\Aff(2,\mathbf C)$.
\begin{lemma}\label{G_Kaffine}
 $G_K$ is the subgroup of $\GL_2(\mathbf R)_K$ acting trivially on $R$. It is 
 isomorphic to $\Aff(1,\mathbf R)$
\end{lemma}
\begin{proof}
 $G_K$ acts on $R$, and since $\Aff(2,\mathbf R)$ has no center, this action
 must be trivial. This implies that $G_K< \GL_2(\mathbf R)_K$. Moreover, any
 $g\in \GL_2(\mathbf R)_K$ commutes with the linear part of any  
 $h\in  \Aff(2,\mathbf R)_{\Delta}$ by definition. Hence $t:=ghg^{-1}h^{-1}$ is a translation.
 If moreover $g$ acts trivially on $R$, also $t$ does, so then $t=0$ and $g\in G_K$.
 That $G_K$ is isomorphic to $\Aff(1,\mathbf R)$ is clear at this point, since
 it consists of matrices that can be written, in the above basis, as
 \[
\begin{pmatrix}
1 & a & \\
0 & b & \\
& & 1 & a\\
& & 0 & b
\end{pmatrix}
\]
and of course $R$ is spanned by $(1,0,0,0)$ and $(0,0,1,0)$.
\end{proof}
We have a commutative diagram of groups with exact rows and columns:
$$\begin{CD}\label{groups}
@. @. @. 1\\
@. @. @. @VVV\\
@. 1 @. @. A\subset \mathbf R^2\\
@. @VVV @. @VVV\\
1 @>>> K @>>> \Gamma @>\rho>> G_A\subset \Aff(2,\mathbf R)\\
@.    @VVV          @ |         @VVV \\ 
1 @>>> K_A @>>> \Gamma @>l\rho>> G\subset \GL_2(\mathbf R)\\
@. @VVV @. @VVV\\
@. A @. @. 1\\
@. @VVV \\
@. 1\\
\end{CD}$$
where: $G_A$ is the image of $\rho$ in $\Aff(2,\mathbf R)_{\Delta}$, which is
identified with $\Aff(2,\mathbf R)$;
$A$ is the subgroup of translations inside $G_A$.

We switch our attention to the action of $G_K$ on $\mathbf C^2$. 
We fix a splitting $\mathbf C^2= R\oplus iR$ as real linear space.
For a point $x\in \mathbf C^2\setminus R$,
there exists a unique (real) line $L_x\subset R$ such that $x$ is contained 
in its complexification, denoted $L_x^{\mathbf C}$.
Note that for any $x\in \mathbf C^2- R$ there is unique real line 
$R_x$ through $x$ which is contained in affine subspace
$x+ iR$ and intersecting $R$.
The complexification of $R_x$ intersects $R$ by a real line 
$L_x$. The former is  also equal to the complexification  $L_x^{\mathbf C}$
of $L_x$.
Note that different lines $ L_x^{\mathbf C},  L_y^{\mathbf C}$ intersect only if
they intersect in $R$.

\begin{lemma}
The set $L_x^{\mathbf C}\setminus L_x$ consists of exactly
two free orbits of $x$ under $G_K$.
\end{lemma}
\begin{proof}

The group $G_K$ is isomorphic to $\Aff(1,\mathbf R)$, and the action
of $g\in G_K$ on $x=(x_1,x_2,x_3,x_4)$ can be easily computed: there exist $a,b$ such that
 $g(x)=(x_1+ax_2,bx_2,x_3+ax_4,bx_4)$. We claim: $L_x=\{(x_1,\lambda x_2,x_3,\lambda x_4),\ \lambda \in \mathbf R\}$.
 Indeed it is a translation by $(x_1,0,x_3,0)$ of a line in $iR$; it touches $R$ for $\lambda=0$; 
 it goes through $x$ at $\lambda=1$. Finally, its
 complexification is clearly $\{(x_1+\mu x_2,\lambda x_2,x_3+\mu x_4,\lambda x_4),\ \lambda,\mu \in \mathbf R\}$.
 At this point the Lemma becomes clear, as the two free orbits of $x$ under $G_K$ 
 can be distinguished by the sign of $b$ above.
\end{proof}
The previous lemma can be rephrased by saying that the quotient space
of $\mathbf C^2\setminus R$ under the action of $G_K$ is canonically
the Grassmannian manifold, $Gr(1,2)$, of affine lines in $R$.
There is a fiber bundle structure $s:Gr(1,2)\to S^1$, where each fiber
parametrizes a family of parallel lines in $R$.
$Gr(1,2)$ is non-orientable and homeomorphic to a Moebius band. The
quotient of $\mathbf C^2\setminus R$ by $G_K^+$, a connected component
of $G_K$, is a cylinder mapping to $Gr(1,2)$ as its canonical orientable
double covering. We denote such covering by $Gr(1,2)^o$.
Trying to lift this setup on $V(=\tilde{X})$ under the unramified developing map, introduced
in equation \ref{developingmap}, $p_{\text{aff}}:V\to \mathbf C^2$, it is natural to
introduce the closed submanifold $V_R:=p_{\text{aff}}^*R$.
A crucial observation is that, since $\rho (\Gamma)$
and $G_K$ commute, the action of $\rho(\Gamma)$ descends naturally to
the orbit space $Gr(1,2)$. 
\begin{lemma}
The action of $G^+_K$ lifts to $V$, it is free on $V\setminus V_R$, and commutes with the action of $\Gamma$.
\end{lemma}
\begin{proof}
The group $G^+_K$ is the exponent of its Lie algebra $\mathfrak g_2^K$,
which is lifted on $V$ from $\mathbf C^2$. Note that vector fields generating 
$\mathfrak g_2^K$ on $\mathbf C^2$ are invariant under $\rho(\Gamma)$, hence
are invariant under $\Gamma$ on $V$. Thus $\mathfrak g_2^K$ on $V$ is also lifting of
a Lie algebra $\mathfrak g_2^K$ on $X$. Since $X$ is compact, $\mathfrak g_2^K$ is integrable
and therefore there is an action of $G_K^+$ on $X$, which lifts to an action on $V$ commuting with $\Gamma$.
Since the action of $G_K^+$ on $\mathbf C^2$ is free away from $R$, the same holds on $V\setminus V_R$.
\end{proof}
Let $S$ denote the quotient of $V\setminus V_R$ by $G_K^+$.
Since $G_K^+$ commutes with $\Gamma$ on $V\setminus V_R$, there is a natural action
of $\Gamma$ on $S$.
The projection $S\to Gr(1,2)^o$ is a $\Gamma$-equivariant local diffeomorphism, and $S$ is an open
real surface. Since $V_R$ is closed in $V$ and $\Gamma$-invariant, its image $X_R$ in $X$
is a finite union of compact, real surfaces.
\begin{remark}
The surface $X\setminus X_R$ is an Eilenberg-MacLane space. Indeed projection
$V\setminus V_R\to S$ is a fibration with contractible fiber $G_K^+$, with base an open surface $S$
having contractible universal covering. Hence the universal covering of $V\setminus V_R$ is contractible,
and the same holds for $X\setminus X_R$.
\end{remark}
Observe that the factorization $S\to S/K\to Gr(1,2)^o$ gives $S/K$ a natural 
structure of open orientable topological surface. 
In particular $K$ acts discontinuously on $S$.\\
We summarize in a commutative diagram:
\begin{equation}\label{diagrammmm}
\begin{CD}
V\setminus V_R @>G_K^+>> S @<<<\tilde{S}\\
@|  @VVV   @| \\
V\setminus V_R @>>> S/K @<<< \tilde{S}\\
@Vp_{\text{aff}}VV @VqVV @V\tilde{q}VV\\
\mathbf C^2\setminus R @>G_K^+>> Gr(1,2)^o @<<< \tilde{Gr(1,2)^o}\\
@. @VVV @VsVV\\
@. S^1@<<< \mathbf R\\
\end{CD}
\end{equation}

Where the rightmost horizontal arrows are universal coverings.
\begin{lemma}
The intersection of $\rho(G)\subset A(2,\mathbf R)$ with the normal subgroup of translation,
is either:
\begin{enumerate}\label{trivial_infinite}
 \item Trivial, or
 \item Infinitely generated dense subgroup of $\mathbf R^2$
\end{enumerate}
\end{lemma}
\begin{proof}
The crucial observation is that the character corresponding to the canonical bundle $K_X$
is a non-trivial morphism $\Gamma/[\Gamma,\Gamma]\to \mathbf R^*$, whose image contains an infinte
cyclic subgroup of finite index.
Hence there is an element $s\in l\rho(\Gamma)$ with det$(s)^2=p/q<1$ and $q\neq 1$.
The intersection $T_{\Gamma}$ of $\rho(\Gamma)$ with translations $\mathbf R^2$ is invariant under $l\rho(\Gamma)$.
If $T_{\Gamma}$ is nontrivial then it is infinitely generated since det$(s)^2$ is not algebraic integer.
Thus it contains arbitrary small translations.  
 Since the representation of $\Gamma$ in $\mathbf R^2$ is the same 
 as $l\rho(\Gamma)$ and hence irreducible , the subgroup
 $T_{\Gamma}\subset \mathbf R^2$ must be dense as well.
 Otherwise its closure in $\mathbf R^2$ is a subgroup isomoprhic to 
 $\mathbf R^1 + \mathbf Z\subset \mathbf R^2$ or $\mathbf R^1\subset \mathbf R^2$ with
 $\mathbf R^1$ being invariant subspace for $l\rho(\Gamma)$. 
 
\end{proof}

Consider the case in which the intersection $\rho(\Gamma)\cap \mathbf R^2=(0)$, namely 
$l\rho(\Gamma)\xrightarrow{\sim} \rho(\Gamma)$.
\begin{lemma}
Assume that $l\rho(\Gamma)$ contains a non-trivial homothety $h$. Then the intersection 
$\rho(\Gamma)\cap \mathbf R^2$ has to be infinitely generated.
\end{lemma}
\begin{proof}
In this case $\rho(\Gamma)$ is isomorphic to $l\rho(\Gamma)$ and hence commutes with $h$.
Thus the image $\rho(\Gamma)$ is contained in Centralizer$(h)\subset A(2,\mathbf R)$ which
is $\text{GL}_2(\mathbf R)_h\subset A(2,\mathbf R)$. Hence the representation $\rho$ is linear,
contradicting remark \ref{nonlinear}.
\end{proof}

We proceed to a detailed study of our surface $S$ and its map to $Gr(1,2)^o$.
Observe that the pre-image of any point in $\mathbf R$ under $s$ is a real line,
corresponding, in Grassmannian, to a family of parallel lines in $\mathbf R^2$.
Therefore $s$ is the universal covering of the angle map.
Moreover, the pre-image of a point in $\mathbf R$ under $s\circ \tilde{q}$ is a union
of open segments, half-lines, or complete lines, and each of them embeds
naturally into the corresponding line in Grassmannian.

Thus we can split our study into two cases: 

\begin{enumerate}
 \item The map $V\setminus V_R\to \mathbf C^2\setminus R$ is an isomorphism and
 hence $\pi_1(V\setminus V_R)$ is a cyclic subgroup of $\tilde A(2,\mathbf R)$
 which is a cyclic universal covering of the affine group  $A(2,\mathbf R)$.
 The quotient $(V\setminus V_R)/G_K^+$ is equal then to $ \tilde Gr(1,2)^o$, i.e cyclic non-ramified
 covering of the Grassmanian.
 \item The map  $V\setminus V_R \to \mathbf C^2\setminus R$ is not an isomorphism
 and hence the map $(V\setminus V_R)/G_K^+= S\to Gr(1,2)^o$ is 
 a $\Gamma$-equivariant map of open surfaces, but not a covering map.
 \end{enumerate}
 We will call the first case classical, and the second 
 case pathological.

 \section{$V_R=\emptyset$}
 In this section we exclude the possibility that $V_R=\emptyset$.
 Let us show how general steps of the proof.\\
 In the classical case group $\Gamma$ must be a discrete subgroup
 of $\tilde A(2,\mathbf R)$.
 If  $V_R=\emptyset$ then the universal covering 
 of $X$ is equal to $\mathbf C^2\setminus R$ and contractible.
 Hence cohomological dimension of $X$ which is equal
 to $4$ must be equal to cohomological dimension of
 $\Gamma < \tilde A(2,\mathbf R)$. The following Theorem \ref{discreteinaffine},
 which is of independent interest, is implemented in Proposition \ref{absurd11}
 to find a contradiction.\\
 In the non-classical case, the induced angle map $\tilde S\to \mathbf R$ must have some
 disconnected fibers. A careful study of their connected components, along with the
 geometry of corresponding stabilizers in $\Gamma$, lead to the required contradiction.

  \begin{theorem} \label{discreteinaffine}
Let  $\Gamma $ be a finitely generated
discrete subgroup of $\tilde A(2,K)\subset \tilde A(2,\mathbf R)$,
$K$ a number field.
Assume that:
\begin{enumerate}
  \item The cohomological dimension of $\Gamma$  is $4$.
  \item The linearization $l(\Gamma)\subset \GL(2,\mathbf R)$
  has no invariant subspace in $\mathbf R^2$.
  \item For any subgroup of finite index $\Gamma'<\Gamma$, the image of the  map $\det : l(\Gamma')\to 
\mathbf R^*$ is infinite cyclic.
 \end{enumerate}
Then the group contains a finite index subgroup $\Gamma'$ which interescts 
the center
$Z\subset \tilde A(2,\mathbf R)$ non-triavially. Moreover the quotient group 
$l(\Gamma')\subset GL(2,\mathbf R)$
sits in an extension: 
\begin{equation}
1\to \pi_g\to  l(\Gamma') \to \mathbf Z\to 1 
\end{equation}
where $\pi_g$ denotes the 
fundamental
group of a compact Riemann surface of genus $ g \geq 2$
  \end{theorem}
\begin{proof}
We start with a:
\begin{lemma}
The above condition 2) implies that the 
intersection $l(\Gamma)_1:=l(\Gamma)\cap \SL(2,\mathbf R)\subset \GL(2,\mathbf R)$
is discrete.
\end{lemma}
\begin{proof}
Assume that $l(\Gamma)_1$ is dense in  $\SL(2,\mathbf R)$ and bring this
assumption to contradiction. The group  $\SL(2,\mathbf R)$
contains an open subset corresponding to elements with complex
complex eigenvalues $\lambda,\bar \lambda $ and $|\lambda|= 1$.
They correspond to rotations with respect to different metrics
on the real space $R = \mathbf R^2$.
Note that the such elements of $l(\Gamma)_1$ 
with $\lambda$ being a root of unity form a finite
 number of conjugacy classes due to the 
fact that the collection of roots of unity contained in a quadratic extension of 
$K$ is finite.
Thus under the above density assumption $l(\Gamma)_1$ contains two 
non-commuting conjugated
elements $h_1,h_2$ of infinite order, whose complex eigenvalues 
$\lambda,\bar \lambda $ are not roots of unity, and
satisfy $|\lambda|= 1$.

In particular if $h_1^N = h_2^N$ for some $N$ then 
$h_1,h_2$
correspond to the same invariant metric on $\mathbf R^2$ and 
represent rotations with coincident centers, in other words they commute.

Thus the fact that  $h_1,h_2$ don't commute means that they
correpsond to rotations with respect to non-proportional
metrics $g_1,g_2$ on $\mathbf R^2$.Let us fix a generator
$c$ of the central subgroup $\mathbf Z\subset \tilde A(2,\mathbf R)$. 
Note that the elements $h_i^n,n\in \mathbf Z,i=1,2$ are dense
in the respective subgroups of rotations $SO(2)_{g_i}\subset A(2,\mathbf R)$
In particular we can find  two sequences 
 of pairs of integers $(n_i^j,m_i^j),n_i^j\neq 0,i=1,2$ such that 
 $H_{i,j}= (h_i)^{n_i^j}c^{m^j_i}$ converges to identity.
 In particular the commutators $H_{1,j}H_{2,j} H_{1,j}^{-1}H_{2,j}^{-1}$ converges
  to idenity in $\tilde A(2,\mathbf R)$, albeit none of them is equal to identity.
 On the other hand  $H_1H_2 H_1^{-1}H_2^{-1}$ is equal
 to the commutator $ h_1^{n_1^j}  h_2^{n_2^j}  h_1^{-n_1^j}  h_2^{-n_1^j}$
 which is contained $l(\Gamma)_1$.Thus we obtain that 
 $\Gamma\subset \tilde A(2,\mathbf R)$ is not discrete.

This proves that $l(\Gamma)_1$ can not be dense.\\
 If the closure of $l(\Gamma)_1$ in  $\SL(2,\mathbf R)$
contains a proper non-trivial connected Lie subgroup $G_0$ then
the latter is uppertriangular or abelian 
and the represenation $l\rho$ is not irreduicble in $\mathbf R^2$
contradicting the the assumption of the theorem.

\end{proof}
\begin{remark}
In our proof we used the fact that $l(\Gamma)$ is contained
in $\GL(2,K)$ for some number field, but a similar result
holds without this assumption.\\
Moreover, an analogue of this result holds for any discrete
  subgroup of the universal covering $\tilde ASp(2n,\mathbf R)$
  of affine extension $ASp(2n,\mathbf R)$ of symplectic linear
   group since rotations constitute an open
   subset in $ASp(2n,R)$.
\end{remark}

\begin{lemma}
$\Gamma$ is  central $\mathbf Z$-extension of the fundamental
  group of a three-dimensional compact manifold, which fibers over $S^1$
  with a compact Riemann surface as fiber.
\end{lemma}
\begin{proof}
Since the image of $\det$ is nontrivial, the intersection
of $\Gamma$ with the subgroup of translations $\xrightarrow{\sim} \mathbf R^2$ must be trivial.
Indeed $\Gamma\cap \mathbf R^2$ is invariant under $l(\Gamma)$, and the action
of det$(l(\Gamma))$ shows that $\Gamma$ cannot be discrete
unless $\Gamma\cap \mathbf R^2=0$.
Since $l(\Gamma)_1$ is discrete by the previous Lemma, it is either free or $\pi_g$ with $g \geq 2$.
Since the kernel of the projection $l(\Gamma)\to l(\Gamma)_1$ is non-trivial,
cohomological dimensions satisfy cd$(l(\Gamma)) =$cd$(l(\Gamma)_1) +1$,
thus cd$(l(\Gamma))$ is either $2$ or $3$.
Since cd$(\Gamma)=4$, necessarily we have: $l(\Gamma)_1= \pi_g$
and $\Gamma$ has non-trivial cyclic kernel under composite projection
$\tilde A(2,\mathbf R)\to A(2,\mathbf R)\to \GL(2,\mathbf R)$. 
Thus the only possibility for $\Gamma$ 
satisfying the assumptions of the theorem
is a central $\mathbf Z$-extension
of the fundamental
group of three-dimensional compact manifold, fibered over $S^1$,
with a compact Riemann surface as a fiber.
\end{proof}
This concludes the proof of the Theorem.
\end{proof}
Now we can exclude the classical case in which $\Gamma$ is discrete.
\begin{proposition}\label{absurd11}
There are no $VII_0$ surfaces with $V_R=\emptyset$, and affine complex structure
inducing an affine representation of the fundamental group which satisfies the assumptions of the previous Theorem.
\end{proposition}
\begin{proof}
Let $X$ be such a surface. Then its fundamental group $\Gamma$ satisfies the conclusion
of the previous Theorem.

Consider the action of the cyclic quotient of $\Gamma$
on the surface group $\pi_g$.
Then a generator, $h$, acts naturally as an element of $Sp(2g,\mathbf Z)$.
Given any eigenvalue $\lambda$ of $h$, also $\bar \lambda$, $\lambda^{-1}$
and $\bar{\lambda}^{-1}$ are eigenvalues of $h$.
Moreover, each of them is an algebraic integer.
For such $\lambda$ we obtain a character $\lambda : \mathbf Z\cdot h\to \mathbf C^*$, and
we have dim$(H^1(\Gamma,\mathbf C_{\lambda}))=$ dim$( H^1(X, \mathscr O(\mathbf C_{\lambda}) ))=1$
Hence $\mathscr O(\mathbf C_{\lambda})$ is equal either $\mathscr O_X$ or $\mathscr O(K_X)$ by lemma \ref{van1}.
\begin{claim}
 We have $\lambda=1$ and $\mathscr O(\mathbf C_{\lambda})=\mathscr O$.
\end{claim}
\begin{proof}
Since $K_X$ is defined by a non-trivial rational character, if
$\mathscr O(\mathbf C_{\lambda})=\mathscr O(K_X)$ then
$\lambda\in \mathbf Q\setminus \{\pm 1\}$, which implies
that $\mathscr O(\mathbf C_{\lambda})=\mathscr O(K_X) =\mathscr O(\mathbf C_{\lambda^{-1}})$,
absurd.
\end{proof}
Thus all eigenvalues of $h$ are $1$, and
hence the cyclic group generated by $h$
on $H^1(\pi_g, Z)= Z^{2g}$ acts as a unipotent 
transformation. Consider the maximal
quotient torsion-free group $M$ of 
$H^1(\pi_g, Z)$ with trivial $h$-action.
Since $h$ is nilpotent $M$ is nontrivial
and isomorphic to $\mathbf Z^l$.
Thus $\pi_1(X)$ surjects onto
a central extension
$M\to M' \to \mathbf Z h$ which is abelian group
of rank $\geq 2$ and hence 
 rk$(H^1(X,\mathbf Q)) \geq 2$ which contradicts
the properties of $VII_0$ surfaces.
This completes the prrof that the image $\tilde \rho:\pi_1(X)\subset \tilde AGL(2,\mathbf R)$
cannot be a discrete subgroup in $ \tilde AGL(2,\mathbf R)$.
\end{proof}

\begin{remark}
The above argument was earlier elaborated by Bombieri,
who proved that the commutator $[\pi_1(X),\pi_1(X)]$ is infinitely
generated. 
\end{remark}
Thus we settled the case $X_R=\emptyset$ and $\rho(\Gamma)$ discrete.\\
In case $\rho(\Gamma)$ is not discrete inside $\tilde A(2,\mathbf R)$,
we study the structure of $\Gamma$ via the 
induced angle map $s:\tilde S\to \mathbf R$.
\begin{remark}
The angle map $s$ cannot have only connected fibers. Indeed each connected component 
of a fiber $s^{-1}(t)$ embedds into $\tilde{Gr(1,2)^o}$, hence if every fiber is connected
then $S$ embedds globally into $\tilde{Gr(1,2)^o}$. But then $V\setminus V_R$
embedds into $\mathbf C^2\setminus R$ and therefore $\rho(\Gamma)$ must be discrete.
\end{remark}
The previous remark implies that the set of connected components of fibers 
$\pi_0(s)$ of the angle map $s$
has a non-trivial structure of topological tree, which we now proceed to analyze.
Let $t\in \mathbf R$ be such that $S_t:=s^{-1}(t)$ is disconnected.
  Then $S_t$ consists of open vertical intervals $\{l_x\}_{x\in \pi_0(S_t)}$ 
  and for each $x\in \pi_0(S_t)$ we can define:
  \begin{definition}
  The right of $x$, denoted by $RI(x)$, is the unique connected component of $s^{-1}(s(x),\infty)$ 
  whose closure contains $x$. Likewise the left of $x$, denoted by $LE(x)$, is the unique connected
  component of $s^{-1}(-\infty, s(x))$ whose closure contains $x$.
  \end{definition}
  In particular we can introduce a partial ordering among the various fiber components:
 \begin{definition}
 Let $l_x$ and $l_y$ be connected components of two fibers of $\phi$.
  We say that $l_x \geq l_y$ if $RI(x)\subseteq  RI(y)$.
  We can aslo say that $l_x \sim l_y$ if $l_x \geq l_y$
  and $LE (y) \subseteq LE (x)$.
  \end{definition}
  This structure is naturally related to tree-like structure
  on the set of vertical intervals ${l_x}$.
  Observe that
  if there two intervals $l_x \sim l_y$, then the union $L_x$ of intervals
  $l_z$ such that $l_z\sim l_x$ provides a fibration under $s$ onto
   an open interval $I\subset \mathbf R$ with generic connected fiber $l_z$.
 If no such fibration exists, then $L_x= l_x$ for any connected fiber component $x$.
  Consider now the action of $\Gamma$ on $S$. For $g\in \Gamma$ 
  it maps $RI(x)$ into $RI(g(x))$.

  In particular for every fiber component $x$, the group $\Gamma$ splits into three subsets:
  \begin{itemize}
  \item $\Gamma^+_x $ consists of those $g\in \Gamma$ such that $g(l_x) \geq l_x$ and $ g(l_z)\neq l_x$
  for any $l_z\in L_x$.
  \item $\Gamma^-_x $ consists of those $g\in \Gamma$ such that $l_x\geq g(l_x) $ and $ g(l_z)\neq l_x$
  for any $l_z\in L_x$.
  \item  $\Gamma^0_x$ is the subgroup
  mapping $L_x$ into itself.
  \end{itemize}
  Note that $\Gamma^+_x$ and $\Gamma^-_x$ are free semigroups.
  The key technical point of our analysis is:
\begin{lemma}
  If for every interval $l_x$ we have $l_x= L_x$
  then there exists $x$ such that $\Gamma^0_x$ is trivial, and there is an injective
  morphism $H_1(\Gamma,\mathbf Z)\to \mathbf R$.  
\end{lemma}
\begin{proof}
  Assume 
  $\Gamma_{x}^0$ is non-trivial for every $x$.
  Observe that the set of eigen-directions of non-scalar elements of $l\rho(\Gamma)$
  is at most countable, hence there exists $x$ such that 
  $\Gamma_{x}^0$ consists either of scalar matrices or of translations.
  If there are scalar matrices then $L_x$ contains a copy of $l_x\times \mathbf R_{>0}$.
  If there are translations then $L_x$ contains a copy of $l_x\times \mathbf R$.
  Either way we obtain a contradiction to our assumption that $l_x=L_x$.\\
  Thus $\Gamma$ is a union of two free semi-groups
  $\Gamma_{x}^+$ and $\Gamma_{x}^-$ which are inverse to each other in $\Gamma$.
  Consider a morphism $h_+ :\Gamma_{x}^+\to \mathbf R$ which maps
  free generators to uncommensurable transcendental numbers. Such morphism obviously exists.
  We can then define $h_-$ on $\Gamma_{x}^-$ as $h_-(g):= h_+ (g^{-1} )$.
  Morphisms $h_+, h_-$ are compatible and define a group morphism  
  $h : \Gamma\to \mathbf R$ whose kernel coincides with $[\Gamma, \Gamma]$.
\end{proof}
\begin{corollary}
  Assumptions as in the previous Lemma, then rk$(H_1(\Gamma, \mathbf Z))\geq 2$ 
\end{corollary}

\begin{proof}
Indeed the proof shows that if rk$=1$ then $\Gamma$ is cyclic, 
which is impossible.
\end{proof}
Observe that the Corollary leads to a contradiction:
since $X$ is a $K(\Gamma,1)$ we know that $H^1(X,\mathbf R)=H^1(\Gamma,\mathbf R)$
and the LHS is 1-dimensional. We deduce that
there must exist uncountably many $x\in \tilde S$ such that $l_x\subsetneq L_x$.
Let us fix one such $x$.
Arguing as in our previous remark it is clear that $L_x$ embedds into $\tilde{G(1,2)^o}$.
This implies moreover that $L_x$ has a non-trivial boundary in $S$, otherwise
$L_x$ would be equal to $S$, implying that $S$
embedds globally into $\tilde{G(1,2)^o}$. Let $\partial_x:=\partial L_x$ be the non-empty
boundary of $L_x$ in $S$. It is a union of $s$-vertical segments, and $s(\partial_x)$
consists of either one or two points. Denote by $\Gamma_x^0$ the stabilizer of $L_x$ in $\Gamma$.
A useful intuitive description of $L_x$ and $\partial_x$ is as follows:
$L_x$ is a maximal subset of $S$ on which $s$ restricts to a fibration with intervals as fibers.
Its boundary $\partial_x$ is the fiber of $s$ in the closure of $L_x$ where the fibration
property breaks down, i.e. $\partial_x$ is not connected and a non-trivial disjoint union of segments.

\begin{lemma}
 The non-trivial stabilizer $\Gamma_x^0$ is a discrete subgroup of $\tilde{A(2,\mathbf R)}$
 and moreover it is contained in a subgroup of transformations preserving one direction in $\mathbf R^2$.
\end{lemma}
\begin{proof}
 Since the $\Gamma$-orbit of $L_x$ is a tree, if $\Gamma_x^0$ were trivial then 
$\Gamma$ would be a free group, which is impossible. Hence $\Gamma_x^0$
is nontrivial, and a subgroup of index at most 2 
acts on the each of the (at most 2) connected components of $\partial_x$.
Hence $\Gamma_x^0$ preserves at least one direction corresponding to a point of $s(\partial_x)$.
Finally since $L_x$ embedds in $\tilde{G(1,2)^o}$ its pre-image in $V\setminus V_R$
embedds in $\mathbf C^2\setminus R$ and hence $\Gamma_x^0$ acts discretely.
\end{proof}
\begin{lemma}
 $\Gamma_x^0$ contains at most a cyclic subgroup of translations.
\end{lemma}
\begin{proof}
 Let $\tau$ be any translation on $\mathbf R^2$, then it preserves the fibers of the
 angular fibration $G(1,2)\to S^1$. If the translation is in 
 $\Gamma_x^0$ then it further preserves the boundary $\partial_x$, which is a union of $s$-vertical segments.
 Observe that if the subgroup of translations in $\Gamma_x^0$ is not cyclic then 
 the orbit of a connected segment in $\partial_x$ is dense in the corresponding fiber of $G(1,2)^o\to S^1$, 
 which means that $\partial_x$ coincides with that fiber.
This is however
 impossible: a neighborhood of such fiber in $S$ is also invariant under the same group 
 of translations, which clearly implies that $\partial_x$ is contained in the interior of $L_x$. 
 Contradiction.
\end{proof}

\begin{lemma}
 $\Gamma_x^0$ is abelian of rank at most 2.
\end{lemma}
\begin{proof}
We know that $l(\Gamma_x^0)$ fixes a direction in $\mathbf R^2$,
which we can assume is $(1,0)$. Hence we can assume
$l(\Gamma_x^0)$ sits in the upper triangular group $UT_2<\GL_2$.\\
 The first step is to show that $\Gamma_x^0$ contains a non-trivial translation.
 In fact we can show a bit more: that $\Gamma_x^0$ contains no linear transformations.
 If $A\in \Gamma_x^0\cap UT_2$, then $A$ acts on the line spanned 
 by $(1,0)$ via multiplication by $a\in \mathbf R^*$. This contradicts the discreteness
 of the action unless $A=\mathbf 1$.\\
 Hence the subgroup of translations is nontrivial and by the previous Lemma cyclic, 
 and we have an exact sequence:
 $$1\to \mathbf Z\tau \to \Gamma_x^0\to l(\Gamma_x^0)\to 1$$
 Upon replacing $\Gamma_x^0$ with a subgroup of index at most 2 we can assume that 
 $\tau$ is in the center. 
 This implies, by direct computation, that every element of
 $l(\Gamma_x^0)$ is the identity along the direction along which $\tau$ translates:
 if $\tau(x)=x+t$ and $A\in l(\Gamma_x^0)$ then $At=t$. We have two cases:\\
 1) The direction $t$ along which $\tau$ translates is not $(1,0)$, and 
 then $l(\Gamma_x^0)< \mathbf R^*$.\\
 2) The direction $t=(1,0)$, and then $l(\Gamma_x^0)< A(1,\mathbf R)$.\\
 Both statements follow by a simple computation combining $A\in UT_2$ with $At=t$.
 Observe that case 2) cannot happen: indeed the induced action of $\Gamma_x^0$ along 
 direction $(1,0)$ 
 is by scalar multiplication
 and hence not discrete.
 Therefore we must be in case 1) which is certainly abelian of rank at most 2.
\end{proof}

Let us show how the fact that $\Gamma_x^0$ is abelian of rank at most 2 implies the Theorem
for the case $V_R=\emptyset$. In fact since $X$ is a $K(\Gamma, 1)$ we know that $\Gamma$
has cohomological dimension cd$(\Gamma)=4$. On the other hand $\Gamma$ acts on the topological tree
$\pi_0(s)$ of fibers of the angle map $s:\tilde S\to \mathbf R$. It follows from Bass-Serre theory
\cite{5} that cd$(\Gamma)\leq \text{cd}(\Gamma_x^0)+1$ where $\Gamma_x^0$ is a stabilizer of a vertex.
However $\Gamma_x^0$ is abelian of rank at most 2, hence its cohomological dimension is at most 2.
Therefore $\Gamma$ has cohomological dimension $\leq 3$. This concludes the proof of the
Theorem in the case $V_R=\emptyset$.

\section{$V_R\neq \emptyset$}\label{four}

In this section we exclude the possibility that
$V_R\neq \emptyset$, thereby proving the Main Theorem.
Before getting into the proof, let us recall that 
the only real surfaces carrying an affine structure
have vanishing Euler characteristic, hence are
the torus, and the Klein bottle, which is covered by a torus.
The affine structures on a torus are classified in \cite{ny}.
In fact they correspond to discrete rank-2 abelian subgroups
inside connected abelian Lie subgroup of the universal cover
$\tilde{A(2,\mathbf R)}$ of the affine group.
There are exactly 
five distinct affine structures on $2$-dimensional real tori:
\label{real affine torai}
\begin{itemize}
 \item [(i)] A lattice in $\mathbf R^2$
 \item [(ii)] A discrete cyclic group in $\mathbf R^2\setminus \{0\}$
 \item [(iii)] A rank-2 discrete subgroup of $\mathbf C\subset \tilde{A(2,\mathbf R)}$.
 Under the universal cover map $\tilde{A(2,\mathbf R)}\to A(2,\mathbf R)$, the image of such rank 2 discrete
 subgroup to a non-discrete rank-2 subgroup of $\mathbf C^*$ 
 \item [(iv)] A lattice in $\mathbf R^*\times \mathbf R^*$
 \item [(v)] A lattice in $\mathbf R\times \mathbf R^*$
\end{itemize}
There is a well defined notion of line in a real torus $T$ with
affine structure. In general such lines are affinely isomorphic either to
complete lines, or half lines, or closed intervals in $\mathbf R^2$.
Note that in case (i) all lines in $T$ correspond to complete lines 
in $\mathbf R^2$. Similarly in cases (ii), (iii) all lines in $T$
which correspond to lines not passing through $0\in \mathbf R^2$
are affinely isomorphic to complete lines, and hence this property
is violated only by a codimension 1 subset of lines in $Gr(1,2)$.
In case (iv) we have complete lines and half lines, while in case (v)
we have half lines and closed intervals.\\
Now we set up some notation we need in the proof.
Denote by $V_R$ the pre-image of $R$ inside $V$, and by $X_R$
 its image in $X$. $X_R$ is a disjoint union of a finite number of real, compact 
 surfaces that inherit an affine structure from $X$. We can 
 assume therefore that every component of $X_R$ is a torus.
 Correspondingly, each connected component of $V_R$ maps
 homeomorphically to $R$, and the image can be $R$, $R\setminus 0$,
 an open half-space $H$, or a quadrant $Q$ - by the above remark on
 affine structures on tori.\\
 Denote by $V_R^i$ the connected components of $V_R$ -
 which we can assume to be real 2-tori - and for each $i$
 let $U_i=\{x\in V $ s.t. $\overline{G_K\cdot x}\cap V_R^i\neq \emptyset \}$. This is
 a $G_K$-invariant open neighborhood of $V_R^i$.
 \begin{lemma}
 Let $x\in V\setminus V_R$ be any point, and $G_K^+\cdot x$ its orbit.
 Then the boundary of the closure of $G_K^+\cdot x$ is a disjoint union of segments,
 which are simultaneously embedded into $R\subset \mathbf C^2$ under $p_{\text{aff}}$.
\end{lemma}
\begin{proof}
Indeed the restriction of $p_{\text{aff}}$ on any $G_K^+$-orbit is an embedding,
while the action of $G_K^+$ on $V\setminus V_R$ is free, and hence the boundary of
such orbits is contained in $V_R$. Such boundary embeds into 
the boundary of the closure of $G_K^+\cdot p_{\text{aff}}(x)\subset \mathbf C^2$,
which is the line $R_{p_{\text{aff}}(x)}$ corresponding to the point $p_{\text{aff}(x)}$.
\end{proof}

\begin{definition}
We define $R^i_x$ the segment of line $R_{p_{\text{aff}}(x)}$ which is contained in $U_i$.
If $x$ is contained in several such $U_i$'s, then the disjoint union of $R^i_x$ embeds
into $R_x$ under $p_{\text{aff}}$.
\end{definition}

First we handle affine structures (iv) and (v). 
Observe that the open domain $U_i$
contains a non-trivial 1-cycle $S^1_i$  which is the boundary of
a transversal disc to $V_R$ in $V$.

\begin{lemma}$[S^1_i]\in \Gamma$ maps to the generating central element
of the group $\tilde{A(2,\mathbf R)}$ under $\rho$.
\end{lemma}
\begin{proof}
Indeed it maps into the generating element of the cyclic group
$\pi_1(\mathbf C^2\setminus \mathbf R^2)$ under the map $V\to \mathbf C^2$.
\end{proof}

\begin{corollary} The pre-image of $S^1_i$ in $\tilde {V \setminus V_R}$
surjects onto $\mathbf R$ under the angle map $s:\tilde S\to \mathbf R$.
\end{corollary}
The next Corollary uses the notation about $s$-vertical lines $l_x$ introduced 
in the previous Section.
\begin{corollary}
If $\tilde{S^1_i}\subset L_x$ for some vertical segment
$l_x\in \tilde S$ then $ L_x= \tilde S$
and  hence $\tilde {V\setminus V_R}$ embeds in $\tilde {\mathbf C^2\setminus \mathbf R^2}$.
\end{corollary}
\begin{proof}
Indeed $L_x$ surjects onto $\mathbf R$ and is simply connected
by the definition $L_x$. It has no boundary in $\tilde S$
and hence it is equal to $\tilde S$.
\end{proof}

\begin{lemma}
In cases (iv) and (v) for any vertical segment $l_x$ in $U_i$ we have $U_i= L_x$. 
\end{lemma}
\begin{proof}
Note that in these cases for $y\in U_i$ the orbit $G_K^+\cdot y$
intersects $V_R$ by segments, half-lines or lines.
Observe that lines in $V_R$ correspond to
subsets of parallel lines in $\mathbf R^2$ intersecting
the fundamental domain of affine structure, which is convex: 
it is a quadrant in case
(iv), and a half-plane in case (v).
In particular whenever a line intersects the fundamental domain,
the intersection is connected.
Now consider the $G_K^+$-orbit of $S^1_i$. Let us denote by $Y_i$ its image
under $V\setminus V_R\to S$.
Observe tha since $G_K^+\cdot y\cap V_R$ is connected for every $y$,
the surjection $Y_i\to S^1$ induced by angle map has connected fibers. This implies that
for any $x\in \tilde Y_i$, there is an inclusion $\tilde{S^1_i}\to L_x$. 
Therefore $L_x\to \mathbf R$ is onto, and the conclusion 
follows by the previous Corollary.
\end{proof}

\begin{lemma} 
 Cases iv) v) cannot occur.
\end{lemma}

\begin{proof}
Indeed $\Gamma$ maps $U_i$ into itself.
In case (iv) we have that $\rho(\Gamma)\subset GL(2,\mathbf R)$
corrsponding to the vertex of the quadrant.
In case (v) we have that $\rho (\Gamma)$ is contained in a triangular
subgroup of $A(2,\mathbf R)$ stabilizing the bounadry line
of the fundamantal domain $U_i$.
Both cases are impossible.
\end{proof}

Next we handle cases (i),(ii),(iii).

\begin{lemma}\label{bound3}
 Assume that $U$ is a domain in $V$ such that $p_{\text{aff}}$ restricts to a covering map 
 $U\to U':=p_{\text{aff}}(U)$
 with the property that dim$(\partial U')<3$. Then $V=\bar U$.
\end{lemma}
\begin{proof}
 If $\bar U\subsetneq V$, then $\partial U$ had dimension 3, and hence the same
 holds for $\partial U'$ since $p_{\text{aff}}$ is a local isomorphism on $U$.
\end{proof}
\begin{lemma}\label{no1,2,3}
The affine structure on $V_R^i$ cannot be of type (i),(ii) or (iii). 
\end{lemma}
\begin{proof}
Note that $U_i$ in the case (i) maps isomorphically to $\mathbf C^2$. 
Similarly in case (ii) $U_i$ maps isomorphically onto
$\mathbf C^2\setminus \{0\}$, and in case (iii) $U_i$ maps as a universal covering of
$\mathbf C^2\setminus \{0\}$, and hence is an isomorphism.
From lemma \ref{bound3} we deduce that in all these cases $U_i=V$.
But then $\Gamma$ stabilizes $V_R^i$, hence $\Gamma$ is almost abelian, {\it i.e.}
contains an abelian subgroup of finite index coming from the action on $V_R^i$.
Hence $l\rho$ is reducible on a finite \'etale cover of $X$, contradicting Lemma \ref{Gnotcyclic}. 
\end{proof}

This finishes the proof of the Main Theorem 
\ref{bogomolov}.


\begin{thebibliography}{ELMNPM}

\bibitem[Al87]{alp} R. Alperin, \textit{An elementary account of Selberg's lemma}, Enseign. Math. (2) 33 (1987), no. 3-4, 269 - 273.
\url{http://www.e-periodica.ch/cntmng?pid=ens-001:1987:33::94}

\bibitem[Bo76]{1} F. Bogomolov, \textit{Classification of surfaces of class $VII_0$ with $b_2=0$}, 
Izvestiya Akademii Nauk SSSR. Seriya Matematicheskaya, 10 (2), 1976, pp. 273-288
\url{http://www.turpion.org/php/paper.phtml?journal_id=im&paper_id=1688}

\bibitem[Bo82]{9} F. Bogomolov, \textit{Surfaces of class $VII_0$ and affine geometry} 
Izv. Akad. Nauk SSSR Ser. Mat. 46 (1982), no. 4, 710-761, 896.
\url{http://iopscience.iop.org/article/10.1070/IM1983v021n01ABEH001640}

\bibitem[BH75]{55} E. Bombieri, D. Husemoller, \textit{Classification and embeddings of surfaces}, Proc. of Symp. in Pure Math. AMS 29, 329-420 (1975)

\bibitem[GS92]{6} M. Gromov, R. Schoen, \textit{Harmonic maps into singular spaces and p-adic
superrigidity for lattices in groups of rank one}, Publications Math\'ematiques IHES, 76 (1992), pp.165 - 246,
\url{http://www.ihes.fr/~gromov/PDF/7%5B85%5D.pdf}

\bibitem[In74]{7} M. Inoue, \textit{On surfaces of Class $VII_0$}, Invent. Math. 24 (1974), 269 - 310
\url{https://link.springer.com/article/10.1007%2FBF01425563}

\bibitem[NY74]{ny} T. Nagano, K. Yagi, \textit{The affine structures on the real two-torus. I},
Osaka J. Math. Vol 11, 1 (1974), 181-210, \url{https://projecteuclid.org/euclid.ojm/1200694718}

\bibitem[LYZ90]{8} Li, Yau and Zheng, \textit{A simple proof of Bogomolov's
theorem on class $VII_0$ surfaces with $b_2=0$}, Illinois J. Math.,34, (1990), 217 - 220
\url{http://projecteuclid.org/euclid.ijm/1255988265}

\bibitem[LYZ94]{10} Li, Yau and Zheng, \textit{On projectively flat Hermitian
manifolds}, Comm. Anal. Geom., 2, (1994), 103 - 109.
\url{http://intlpress.com/site/pub/files/_fulltext/journals/cag/1994/0002/0001/CAG-1994-0002-0001-a006.pdf}

\bibitem[Se80]{5} J. P. Serre, \textit{Trees}, Springer Verlag, 1980. 
\url{http://www.springer.com/us/book/9783540442370}

\bibitem[Te94]{4} A.D. Teleman, \textit{Projectively flat surfaces
 and Bogomolov's theorem on class $VII_0$ surfaces}. Internat. J. Math. 5 (1994), no. 2, 253 - 264
 \url{http://www.worldscientific.com/doi/abs/10.1142/S0129167X94000152}

\end{thebibliography}
\end{document}